\documentclass[11pt]{amsart}
\usepackage{amscd}
\usepackage{amsmath}
\usepackage{amsxtra}
\usepackage{amsfonts}
\usepackage{amssymb}

\newtheorem{theoremDutkeySilvestrov}{Theorem} 
\newtheorem{corollaryDutkeySilvestrov}[theoremDutkeySilvestrov]{Corollary}
\newtheorem{lemmaDutkeySilvestrov}[theoremDutkeySilvestrov]{Lemma}
\newtheorem{propositionDutkeySilvestrov}[theoremDutkeySilvestrov]{Proposition}

\newtheorem{definitionDutkeySilvestrov}[theoremDutkeySilvestrov]{Definition}
\newtheorem{remarkDutkeySilvestrov}[theoremDutkeySilvestrov]{Remark}

\newtheorem{exampleDutkeySilvestrov}[theoremDutkeySilvestrov]{Example}

\renewcommand{\theclaimDutkeySilvestrov}{\textup{\theclaimDutkeySilvestrov}}

\def\HDutkaySilv{\mathcal{H}}
\def\KDutkaySilv{\mathcal{K}}
\newcommand{\ipDutkaySilv}[2]{\left\langle #1\, , \,#2\right\rangle}
\newcommand{\brDutkaySilv}{\mathbb{R}}
\newcommand{\bcDutkaySilv}{\mathbb{C}}
\newcommand{\btDutkaySilv}{\mathbb{T}}
\newcommand{\bnDutkaySilv}{\mathbb{N}}
\newcommand{\bzDutkaySilv}{\mathbb{Z}}

\begin{document}

\title[Reducibility of the wavelet representation]{Reducibility of the wavelet representation associated to the Cantor set}
\author{Dorin Ervin Dutkay}
\address{[Dorin Ervin Dutkay] University of Central Florida\\
	Department of Mathematics\\
	4000 Central Florida Blvd.\\
	P.O. Box 161364\\
	Orlando, FL 32816-1364\\
U.S.A.\\} \email{ddutkay@mail.ucf.edu}
\author{Sergei Silvestrov}
\address{[Sergei Silvestrov] Centre for Mathematical Sciences\\
Lund University\\
Box 118, SE-221 00 Lund, Sweden}
\email{ssilvest@maths.lth.se}
\thanks{Research supported in part by The Swedish Foundation for International Cooperation in Research and Higher Education
(STINT), The Swedish Research Council, The Swedish Royal Academy of Sciences and The Crafoord Foundation. }

\begin{abstract}
We answer a question by Judith Packer about the irreducibility of the wavelet representation associated to the Cantor set. We prove that if the QMF filter does not have constant absolute value, then the wavelet representations is reducible.
\end{abstract}

\maketitle


\section{Introduction}
Wavelet representations were introduced in \cite{Jor01b,Dut02,DuJo07} in an attempt to apply the multiresolution techniques of wavelet theory \cite{Dau92} to a larger class of problems where self-similarity, or refinement is the central phenomenon. They were used to construct wavelet bases and multiresolutions on fractal measures and Cantor sets \cite{DuJo06b} or on solenoids \cite{Dut06}.

Wavelet representations can be defined axiomatically as follows: let $X$ be a compact metric space and let $r:X\rightarrow X$ be a Borel measurable function which is onto and finite-to-one, i.e., $0<\#r^{-1}(x)<\infty$ for all $x\in X$. Let $\mu$ be a {\it strongly invariant measure} on $X$, i.e.
\begin{equation}
\int_X f\,d\mu=\int_X\frac{1}{\#r^{-1}(x)}\sum_{r(y)=x}f(y)\,d\mu(x),\quad(f\in L^\infty(X))
\label{eqsi}
\end{equation}

Let $m_0\in L^\infty(X)$ be a {\it QMF filter}, i.e.,
\begin{equation}
\frac{1}{\#r^{-1}(x)}\sum_{r(y)=x}|m_0(y)|^2=1\mbox{ for $\mu$-a.e. $x\in X$}
\label{eqqmf}
\end{equation}

\begin{theoremDutkeySilvestrov}\label{thi1}\cite{DuJo07}
There exists a Hilbert space $\HDutkaySilv$, a unitary operator $U$ on $\HDutkaySilv$, a representation $\pi$ of $L^\infty(X)$ on $\HDutkaySilv$ and an element $\varphi$ of $\HDutkaySilv$ such that
\begin{enumerate}
\item
(Covariance) $U\pi(f)U^{-1}=\pi(f\circ r)$ for all $f\in L^\infty(X)$.
\item (Scaling equation) $U\varphi=\pi(m_0)\varphi$
\item (Orthogonality) $\ipDutkaySilv{\pi(f)\varphi}{\varphi}=\int f\,d\mu$ for all $f\in L^\infty(X)$.
\item (Density) $\{U^{-n}\pi(f)\varphi\,|\,n\in\bnDutkaySilv, f\in L^\infty(X)\}$ is dense in $\HDutkaySilv$.

\end{enumerate}
Moreover they are unique up to isomorphism.

\end{theoremDutkeySilvestrov}

\begin{definitionDutkeySilvestrov}\label{defi3}
We say that $(\HDutkaySilv,U,\pi,\varphi)$ in Theorem \ref{thi1} is the {\it wavelet representation} associated to $m_0$.
\end{definitionDutkeySilvestrov}

Our main focus will be the irreducibility of the wavelet representations.

The most familiar wavelet representation is the classical one on $L^2(\brDutkaySilv)$, where $U$ is the operator of dilation by 2, and $\pi$ is obtained by applying the Borel functional calculus to the translation operator $T$, i.e. $\pi(f)=f(T)$ for $f$ bounded function on $\btDutkaySilv$ - the unit circle. This representation is associated to the map $r(z)=z^2$ on $\btDutkaySilv$, the measure $\mu$ is just the Haar measure on the circle, and $m_0$ can be any low-pass QMF filter which produces an orthogonal scaling function (see \cite{Dau92}). For example, one can take the Haar filter $m_0(z)=(1+z)/\sqrt2$ which produces the Haar scaling function $\varphi$.

This representation is reducible; its commutant was computed in \cite{HaLa00} and the direct integral decomposition was presented in \cite{LPT01}.

Some low-pass filters, such as the stretched Haar filter $m_0(z)=(1+z^3)/\sqrt2$ give rise to non-orthogonal scaling functions. In this case super-wavelets appear, and the wavelet representation is realized on a direct sum of finitely many copies of $L^2(\brDutkaySilv)$. See \cite{BDP05}. This representation is also reducible and its direct integral decomposition is similar to the one for $L^2(\brDutkaySilv)$. See \cite{BDP05,Dut06}.

When one takes the QMF filter $m_0=1$ the situation is very different. As shown in \cite{Dut06}, the representation can be realized on a solenoid and in this case it is irreducible. The result holds even for more general maps $r$, if they are ergodic (see \cite{DLS09}).

The general theory of the decomposition of wavelet representations into irreducible components was given in \cite{Dut06}, but there is a large class of examples where it is not known wheather these representations are irreducible or not.

One interesting example, introduced in \cite{DuJo07b}, is the following: take the map $r(z)=z^3$ on the unit circle $\btDutkaySilv$ with the Haar measure $\mu$. Consider the QMF filter $m_0(z)=(1+z^2)/\sqrt2$. The wavelet representation associated to this data is strongly connected to the middle-third Cantor set. It can be realized as follows:

Let $\mathbf C$ be the middle-third Cantor set. Let
$$\mathcal R:=\bigcup\left\{\mathbf C+\frac{k}{3^n}\,|\,k,n\in\bzDutkaySilv\right\}.$$
Let $\HDutkaySilv^s$ be the Hausdorff measure of dimension $s:=\log_32$, i.e., the Hausdorff dimension of the Cantor set. Restrict $\HDutkaySilv^s$ to the the set $\mathcal R$. Consider the Hilbert space $\HDutkaySilv:=L^2(\mathcal R,\HDutkaySilv^s)$. Define the unitary operators on $\HDutkaySilv$:
$$Uf(x)=\frac{1}{\sqrt2}f\left(\frac x3\right),\quad Tf(x)=f(x-1)$$
and define the representation $\pi$ of $L^\infty(\btDutkaySilv)$ on $\HDutkaySilv$, by applying Borel functional calculus to the operator $T$: $\pi(f)=f(T)$ for $f\in L^\infty(X)$.

The scaling function is defined as the characteristic function of the Cantor set $\varphi:=\chi_{\mathbf C}$.

Then $(\HDutkaySilv,U,\pi,\varphi)$ is the wavelet representation associated to the QMF filter $m_0(z)=(1+z^2)/\sqrt2$.

In February 2009, at the FL-IA-CO-OK workshop in Iowa City, following investigations into general multiresolution theories \cite{bfmp2,bfmp1,aijln,ijkln}, Judith Packer formulated the following question: is this representation irreducible?
We will answer this question here, and show that the representation is {\it not} irreducible. Indeed, we show that $m_0=1$ is an exception, and, under some mild assumptions, all the other QMF filters give rise to reducible representations.

In \cite{DLS09}, several equivalent forms of this problem were presented in terms of refinement equations, fixed points of transfer operators or ergodic shifts on solenoids. Using the results in \cite{DLS09} we obtain as a corollary non-trivial solutions to all these problems.

\section{Main Result}
\begin{theoremDutkeySilvestrov}\label{th1}
Suppose $r:(X,\mu)\rightarrow (X,\mu)$ is ergodic. Assume $|m_0|$ is not constant $1$ $\mu$-a.e., non-singular, i.e., $\mu(m_0(x)=0)=0$, and $\log |m_0|^2$ is in $L^1(X)$. Then the wavelet representation $(\HDutkaySilv,U,\pi,\varphi)$ is reducible.
\end{theoremDutkeySilvestrov}

\begin{proof}

We recall some facts from \cite{DuJo07}. The wavelet represntation can be realized on a solenoid as follows:
Let
\begin{equation}
X_\infty:=\left\{(x_0,x_1,\dots)\in X^{\bnDutkaySilv}\,|\, r(x_{n+1})=x_n\mbox{ for all }n\geq 0\right\}
\label{eqi4_1}
\end{equation}
We call $X_\infty$ the {\it solenoid} associated to the map $r$.

On $X_\infty$ consider the $\sigma$-algebra generated by cylinder sets.
Define the map $r_\infty:X_\infty\rightarrow X_\infty$ as follows
\begin{equation}
r_\infty(x_0,x_1,\dots)=(r(x_0),x_0,x_1,\dots)\mbox{ for all }(x_0,x_1,\dots)\in X_\infty
\label{eqi4_2}
\end{equation}
Then $r_\infty$ is a measurable automorphism on $X_\infty$.

Define $\theta_0:X_\infty\rightarrow X$,
\begin{equation}
\theta_0(x_0,x_1,\dots)=x_0.
\label{eqi4_3}
\end{equation}
The measure $\mu_\infty$ on $X_\infty$ will be defined by constructing some path measures $P_x$ on the fibers $\Omega_x:=\{(x_0,x_1,\dots)\in X_\infty\,|\, x_0=x\}.$

Let
$$c(x):=\#r^{-1}(r(x)),\quad W(x)=|m_0(x)|^2/c(x),\quad(x\in X).$$
Then
\begin{equation}
\sum_{r(y)=x}W(y)=1,\quad(x\in X)
\label{eqi4_4}
\end{equation}
$W(y)$ can be thought of as the trasition probability from $x=r(y)$ to one of its roots $y$.

For $x\in X$, the path measure $P_x$ on $\Omega_x$ is defined on cylinder sets by
\begin{equation}
P_x(\{(x_n)_{n\geq0}\in\Omega_x\,|\, x_1=z_1,\dots,x_n=z_n\})=W(z_1)\dots W(z_n)
\label{eqi4_5}
\end{equation}
for any $z_1,\dots, z_n\in X$.

This value can be interpreted as the probability of the random walk to go from $x$ to $z_n$ through the points $x_1,\dots,x_n$.

Next, define the measure $\mu_\infty$ on $X_\infty$ by
\begin{equation}
\int f\,d\mu_\infty=\int_{X}\int_{\Omega_x}f(x,x_1,\dots)\,dP_{x}(x,x_1,\dots)\,d\mu(x)
\label{eqi4_7}
\end{equation}
for bounded measurable functions on $X_\infty$.

Consider now the Hilbert space $\HDutkaySilv:=L^2(\mu_\infty)$. Define the operator
\begin{equation}
U\xi=(m_0\circ\theta_0)\,\xi\circ r_\infty,\quad(\xi\in L^2(X_\infty,\mu_\infty))
\label{eqi4_8}
\end{equation}

Define the representation of $L^\infty(X)$ on $\HDutkaySilv$
\begin{equation}
\pi(f)\xi=(f\circ\theta_0)\,\xi,\quad(f\in L^\infty(X),\xi\in L^2(X_\infty,\mu_\infty))
\label{eqi4_9}
\end{equation}

Let $\varphi=1$ the constant function $1$.

\begin{lemmaDutkeySilvestrov}\label{thi5}\cite{DuJo07}
Suppose $m_0$ is non-singular, i.e., $\mu(\{x\in X\,|\, m_0(x)=0\})=0$. Then the data $(\HDutkaySilv,U,\pi,\varphi)$ forms the wavelet representation associated to $m_0$.
\end{lemmaDutkeySilvestrov}

We proceed to the proof of our main result.

From the QMF relation and the strong invariance of $\mu$ we have
$$\int_X|m_0|^2\,d\mu=\int_X\frac{1}{\# r^{-1}(x)}\sum_{r(y)=x}|m_0(y)|^2\,d\mu=1.$$

By Jensen's inequality we have

$$a:=\int_X \log |m_0|^2\,d\mu\leq \log\int_X|m_0|^2\,d\mu=0.$$
Since $\log$ is strictly concave, and $|m_0|^2$ is not constant $\mu$-a.e., it follows that the inequality is strict, and $a<0$.

Since $r$ is ergodic, applying Birkoff's ergodic theorem, we obtain that
$$\lim_{n\rightarrow\infty}\frac{1}{n}\sum_{k=0}^{n-1}\log|m_0\circ r^k|^2=\int_X\log|m_0|^2\,d\mu=a,\,\mu-\mbox{ a.e.}$$
This implies that
$$\lim_{n\rightarrow\infty}\left(|m_0(x)m_0(r(x))\dots m_0(r^{n-1}(x))|^2\right)^{1/n}=e^a<1,\,\mu-\mbox{a.e.}$$
Take $b$ with $e^a<b<1$.

By Egorov's theorem, there exists a measurable set $A_0$, with $\mu_\infty(A_0)>0$, such that $(|m_0(x)m_0(r(x))\dots m_0(r^{n-1}(x))|^2)^{1/n}$ converges uniformly to $e^a$ on $A_0$. This implies that there exists an $n_0$ such for all $m\geq n_0$:
$$\left(|m_0(x)m_0(r(x))\dots m_0(r^{m-1}(x))|^2\right)^{1/m}\leq b\mbox{ for }x\in A_0$$
so
\begin{equation}
|m_0(x)m_0(r(x))\dots m_0(r^{m-1}(x))|^2\leq b^m,\mbox{ for $m\geq n_0$ and all $x\in A_0$.}
\label{eqt1}
\end{equation}

Next, given $m\in\bnDutkaySilv$, we compute the probability of a sequence $(z_n)_{n\in\bnDutkaySilv}\in X_\infty$ to have $z_m\in A_0$. We have, using the strong invariance of $\mu$:
$$P(z_m\in A_0)=\mu_\infty\left(\{(z_n)_n\,|\, z_m\in A_0\}\right)=\int_{X_\infty}\chi_{A_0}\circ\theta_m\,d\mu_\infty$$
$$=\int_X\frac{1}{\#r^{-m}(z_0)}\sum_{r(z_1)=z_0,\dots, r(z_m)=z_{m-1}}|m_0(z_1)|^2\dots |m_0(z_m)|^2\chi_{A_0}(z_m)\,d\mu(z_0)$$
$$=\int_X|m_0(z_m)m_0(r(z_m))\dots m_0(r^{m-1}(z_m))|^2\chi_{A_0}(z_m)\,d\mu(z_m)$$
$$=\int_X|m_0(x)m_0(r(x))\dots m_0(r^{m-1}(x))|^2\chi_{A_0}(x)\,d\mu(x).$$

Then
$$\sum_{m=1}^\infty P(z_m\in A_0)=\sum_{m\geq1}\int_X |m_0(x)m_0(r(x))\dots m_0(r^{m-1}(x))|^2\chi_{A_0}\,d\mu(x)<\infty$$
and we used \eqref{eqt1} in the last inequality.

Now we can use Borel-Cantelli's lemma, to conclude that the probability that $z_m\in A_0$ infinitely often is zero. Thus, for $\mu_\infty$-a.e. $z:=(z_n)_n$, there exists $k_z$ (depending on the point) such that $z_n\not\in A_0$ for $n\geq k_z$.

Suppose now the representation is irreducible. Then $r_\infty$ is ergodic on $(X_\infty,\mu_\infty)$. So $r_\infty^{-1}$ is too. Using Birkhoff's ergodic theorem it follows that, $\mu_\infty$-a.e.,
\begin{equation}
\lim_{n\rightarrow\infty}\frac{1}{n}\sum_{k=0}^{n-1}(\chi_{A_0}\circ\theta_0)\circ r_\infty^{-k}=\int_{X_\infty}\chi_{A_0}\circ\theta_0\,d\mu_\infty=\mu(A_0)>0
\label{eqt2}
\end{equation}

But
$$\left[(\chi_{A_0}\circ\theta_0)\circ r_\infty^{-k}\right](z_n)_n=\chi_{A_0}(z_k)=0,\mbox{ for $k\geq k_z$}.$$
Therefore the sum on the left of \eqref{eqt2} is bounded by $k_z$ so the limit is zero, a contradiction. Thus the representation has to be reducible.
\end{proof}

Using the results from \cite{DLS09}, we obtain that there are non-trivial solutions to refinement equations and non-trivial fixed points for transfer operators:

\begin{corollaryDutkeySilvestrov}\label{cor2}
Let $m_0$ be as in Theorem \ref{th1} and let $(\HDutkaySilv,U,\pi,\varphi)$ be the associated wavelet representation. Then

\begin{enumerate}
\item
There exist solutions $\varphi'\in\HDutkaySilv$ for the scaling equation $U\varphi'=\pi(m_0)\varphi'$ which are not constant multiples of $\varphi$.
\item
There exist non-constant, bounded fixed points for the transfer operator
$$R_{m_0}f(x)=\frac{1}{\#r^{-1}(x)}\sum_{r(y)=x}|m_0(y)|^2f(y),\quad(f\in L^\infty(X),x\in X).$$
\end{enumerate}
\end{corollaryDutkeySilvestrov}

\begin{remarkDutkeySilvestrov}\label{rem2.4}
As shown in \cite{DuJo07}, operators in the commutant of $\{U, \pi\}$ are multiplication operators $M_g$, with $g\in L^\infty(X_\infty,r_\infty)$ and $g=g\circ r_\infty$. Therefore, if $\KDutkaySilv$ is a subspace which is invariant for $U$ and $\pi(f)$ for all $f\in L^\infty(X)$, then the orthogonal projection onto $\KDutkaySilv$ is an operator in the commutant and so it corresponds to a multiplication by a characteristic function $\chi_A$, where $A$ is an invariant set for $r_\infty$, i.e.,
$A=r_\infty^{-1}(A)=r_\infty(A)$, $\mu_\infty$-a.e., and $\KDutkaySilv=L^2(A,\mu_\infty)$

In conclusion the study of invariant spaces for the wavelet representation $\{U,\pi\}$ is equivalent to the study of the invariant sets for the dynamical system $
r_\infty$ on $(X_\infty,\mu_\infty)$.

\end{remarkDutkeySilvestrov}

\begin{propositionDutkeySilvestrov}\label{pr2.5}
Under the assumptions of Theorem \ref{th1}, there are no finite dimensional invariant subspaces for the wavelet representation.
\end{propositionDutkeySilvestrov}

\begin{proof}
We reason by contradiction. Suppose $\KDutkaySilv$ is a finite dimensional invariant subspaces. Then, as in remark \ref{rem2.4}, this will correspond to a set $A$ invariant under $r_\infty$, $\KDutkaySilv=L^2(A,\mu_\infty)$. But if $\KDutkaySilv$ is finite dimensional then $A$ must contain only atoms. Let $(z_n)_{n\in\bnDutkaySilv}$ be such an atom.
We have
$$0<\mu_\infty((z_n)_{n\in\bnDutkaySilv})=\mu(z_0)P_{z_0}((z_n)_{n\in\bnDutkaySilv}),$$
so $z_0$ is an atom for $\mu$. Since $\mu$ is strongly invariant for $\mu$, it follows that it is also invariant for $\mu$.
Then $\mu(r(z_0))=\mu(r^{-1}(r(z_0)))\geq \mu(z_0)$. By induction, $\mu(r^{n+1}(z_0))\geq\mu(r^n(z_0))$. Since $\mu(X)<\infty$ and $\mu(z_0)>0$ this implies that
at for some $n\in\bnDutkaySilv$ and $p>0$ we have $r^{n+p}(z_0)=r^n(z_0)$. We relabel $r^n(z_0)$ by $z_0$ so we have $r^p(z_0)=z_0$ and $\mu(z_0)>0$.

Since $\mu$ is invariant for $r$ we have $\mu(z_0)\leq \mu(r^{-p}(z_0))=\mu(z_0)$, and this shows that all the points in $r^{-p}(z_0)$ except $z_0$ have measure $\mu$ zero. The same can be said for $r(z_0), \dots, r^{p-1}(z_0)$. But then $C:=\{z_0,r(z_0),\dots, r^{p-1}(z_0)\}$ is invariant for $r$, $\mu$-a.e., and has positive measure. Since $r$ is ergodic this shows that $C=X$, $\mu$-a.e., and so we can consider that $\# r^{-1}(x)=1$ for $\mu$-a.e. $x\in X$. And then the QMF condition implies that $|m_0|=1$ $\mu$-a.e., which contradicts the assumptions in the hypothesis.
\end{proof}

\section{Examples}

\begin{exampleDutkeySilvestrov}\label{ex1}
Consider the map $r(z)=z^2$ on the unit circle $\btDutkaySilv=\{z\in\bcDutkaySilv : |z|=1\}$. Let $\mu$ be the Haar measure on $\btDutkaySilv$. Let $m_0(z)=\frac{1}{\sqrt{2}}(1+z)$ be the Haar low-pass filter, or any filter that generates an orthonormal scaling function in $L^2(\brDutkaySilv)$ (see \cite{Dau92}). Then the wavelet representation associated to $m_0$ can be realized on the Hilbert space $L^2(\brDutkaySilv)$. The dilation operator is
$$U\xi(x)=\frac{1}{\sqrt{2}}\xi\left(\frac x2\right),\quad(x\in\brDutkaySilv,\xi\in L^2(\brDutkaySilv))$$
The representation $\pi$ of $L^\infty(\btDutkaySilv)$ is constructed by applying Borel functional calculus to the translation operator
$$T\xi(x)=\xi(x-1),\quad(x\in\brDutkaySilv,f\in L^2(\brDutkaySilv)),$$
$$\pi(f)=f(T),\quad(f\in L^\infty(\brDutkaySilv)),$$
in particular
$$\pi\left(\sum_{k\in\bzDutkaySilv}a_kz^k\right)=\sum_{k\in\bzDutkaySilv}a_kT^k,$$
for any finitely supported sequence of complex numbers $(a_k)_{k\in\bzDutkaySilv}$.

The Fourier transform of the scaling function is given by an infinite product (\cite{Dau92}):
$$\widehat\varphi(x)=\prod_{n=1}^\infty m_0\left(\frac x{2^n}\right),\quad(x\in\brDutkaySilv).$$

The commutant of this wavelet representation can be explicitely computed (see \cite{HaLa00}): let $\mathcal F$ be the Fourier transform. An operator $A$ is in the commutant $\{U,\pi\}'$ of the wavelet representation if and only if its Fourier transform $\widehat A:=\mathcal F A\mathcal F^{-1}$ is a multiplication operator by a bounded, dilation invariant function, i.e., $\widehat A=M_f$, with $f\in L^\infty(\brDutkaySilv)$, $f(2x)=f(x)$, for a.e. $x\in\brDutkaySilv$. Here
$$M_f\xi=f\xi,\quad(\xi\in L^2(\brDutkaySilv)).$$

Thus, invariant subspaces correspond, through the Fourier transform, to sets which are invariant under dilation by 2.

The measure $\mu_\infty$ on the solenoid $\btDutkaySilv_\infty$ can also be computed, see \cite{Dut06}. It is supported on the embedding of $\brDutkaySilv$ in the solenoid $\btDutkaySilv_\infty$. The path measures $P_x$ are in this case atomic.

The direct integral decomposition of the wavelet representation was described \cite{LPT01}.

For the low-pass filters that generate non-orthogonal scaling function, such as the stretched Haar filter $m_0(z)=\frac{1}{\sqrt{2}}(1+z^3)$, the wavelet representation can be realized in a finite sum of copies of $L^2(\brDutkaySilv)$. These filters correspond to super-wavelets, and the computation of the commutant, of the measure $\mu_\infty$ and the direct integral decomposition of the wavelet representation can be found in \cite{BDP05, Dut06}.

\end{exampleDutkeySilvestrov}

\begin{exampleDutkeySilvestrov}\label{ex2}
Let $r(z)=z^N$, $N\in\bnDutkaySilv$, $N\geq 2$ on the unit circle $\btDutkaySilv$ and let $m_0(z)=1$ for all $z\in\btDutkaySilv$. In this case (see \cite{Dut06}) the wavelet representation can be realized on the solenoid $\btDutkaySilv_\infty$ and the measure $\mu_\infty$ is just the Haar measure on the solenoid $\btDutkaySilv_\infty$, and the operators $U$, $\pi$ are defined above in the proof of Theorem \ref{th1}. For this particular wavelet representation the commutant is trivial, so the representation is {\it irreducible}. It is interesting to see that, by Theorem \ref{th1}, just any small perturbation of the constant function $m_0=1$ will generate a {\it reducible} wavelet representation.
\end{exampleDutkeySilvestrov}

\begin{exampleDutkeySilvestrov}\label{ex3}
We turn now to the example in Judy Packer's question. $r(z)=z^3$ on $\btDutkaySilv$ with the Haar measure, and $m_0(z)=\frac{1}{\sqrt{2}}(1+z^2)$. As we explained in the introduction, this low-pass filter generates a wavelet representation involving the middle third Cantor set. See \cite{DuJo06b} for details. We know that $r(z)=z^3$ is an ergodic map and it is easy to see that the function $m_0$ satisfies the hypotheses of Theorem \ref{th1}. Actually, an application of Jensen's formula to the analytic function $m_0^2$ shows that
$$\int_{\btDutkaySilv}\log|m_0|^2\,d\mu=-2\pi \log 2.$$

Thus, by Theorem \ref{th1}, it follows that this wavelet representation is reducible. However, the problem of {\it constructing} the operators in the commutant of the wavelet representation remains open for analysis.
\end{exampleDutkeySilvestrov}

\newcommand{\etalchar}[1]{$^{#1}$}
\def\cprime{$'$}


\begin{thebibliography}{aaaaaaaaaa}

\bibitem[BDP05]{BDP05}
Stefan Bildea, Dorin~Ervin Dutkay, and Gabriel Picioroaga.
\newblock M{RA} super-wavelets.
\newblock {\em New York J. Math.}, 11:1--19 (electronic), 2005.

\bibitem[BFMP09a]{bfmp1}
Lawrence~W. Baggett, Veronika Furst, Kathy~D. Merrill, and Judith~A. Packer.
\newblock Generalized filters, the low-pass condition, and connections to
  multiresolution analyses.
\newblock {\em J. Funct. Anal.}, 257(9):2760--2779, 2009.

\bibitem[BFMP09b]{bfmp2}
L.W. Baggett, V.~Furst, K.D. Merrill, and J.~A. Packer.
\newblock Classification of generalized multiresolution analyses.
\newblock {\em Preprint}, 2009.

\bibitem[BLM09]{ijkln}
Lawrence~W. Baggett, Nadia~S. Larsen, Kathy~D. Merrill, Judith~A. Packer, and
  Iain Raeburn.
\newblock Generalized multiresolution analyses with given multiplicity
  functions.
\newblock {\em J. Fourier Anal. Appl.}, 15(5):616--633, 2009.

\bibitem[BLP09]{aijln}
L.W. Baggett, N.S. Larsen, J.A. Packer, I.~Raeburn, , and A~Ramsay.
\newblock Direct limits, {MRA}'s and wavelets.
\newblock {\em Preprint}, 2009.

\bibitem[Dau92]{Dau92}
Ingrid Daubechies.
\newblock {\em Ten lectures on wavelets}, volume~61 of {\em CBMS-NSF Regional
  Conference Series in Applied Mathematics}.
\newblock Society for Industrial and Applied Mathematics (SIAM), Philadelphia,
  PA, 1992.

\bibitem[DJ06]{DuJo06b}
Dorin~E. Dutkay and Palle E.~T. Jorgensen.
\newblock Wavelets on fractals.
\newblock {\em Rev. Mat. Iberoam.}, 22(1):131--180, 2006.

\bibitem[DJ07a]{DuJo07b}
Dorin~Ervin Dutkay and Palle E.~T. Jorgensen.
\newblock Fourier frequencies in affine iterated function systems.
\newblock {\em J. Funct. Anal.}, 247(1):110--137, 2007.

\bibitem[DJ07b]{DuJo07}
Dorin~Ervin Dutkay and Palle E.~T. Jorgensen.
\newblock Martingales, endomorphisms, and covariant systems of operators in
  {H}ilbert space.
\newblock {\em J. Operator Theory}, 58(2):269--310, 2007.

\bibitem[DLS09]{DLS09}
Dorin~Ervin Dutkay, David~R. Larson, and Sergei Silvestrov.
\newblock Irreducible wavelet representations and ergodic automorphisms on
  solenoids.
\newblock {\em preprint}, 2009.

\bibitem[Dut02]{Dut02}
Dorin~Ervin Dutkay.
\newblock Harmonic analysis of signed {R}uelle transfer operators.
\newblock {\em J. Math. Anal. Appl.}, 273(2):590--617, 2002.

\bibitem[Dut06]{Dut06}
Dorin~Ervin Dutkay.
\newblock Low-pass filters and representations of the {B}aumslag {S}olitar
  group.
\newblock {\em Trans. Amer. Math. Soc.}, 358(12):5271--5291 (electronic), 2006.

\bibitem[HL00]{HaLa00}
Deguang Han and David~R. Larson.
\newblock Frames, bases and group representations.
\newblock {\em Mem. Amer. Math. Soc.}, 147(697):x+94, 2000.

\bibitem[Jor01]{Jor01b}
Palle E.~T. Jorgensen.
\newblock Ruelle operators: functions which are harmonic with respect to a
  transfer operator.
\newblock {\em Mem. Amer. Math. Soc.}, 152(720):viii+60, 2001.

\bibitem[LPT01]{LPT01}
Lek-Heng Lim, Judith~A. Packer, and Keith~F. Taylor.
\newblock A direct integral decomposition of the wavelet representation.
\newblock {\em Proc. Amer. Math. Soc.}, 129(10):3057--3067 (electronic), 2001.

\end{thebibliography}
\end{document}